\documentclass[pagesize,pdftex,reqno]{amsart}
\usepackage{amssymb,amsfonts,amsmath,amsthm}
\usepackage{graphicx}
\usepackage{enumitem}
\usepackage{autonum}

  { \end{itemize} }
  { \end{enumerate} }

\newcommand{\OL}[0]{\overline} %
\newcommand{\FA}[1]{\ensuremath{\forall #1}}

\newcommand{\HW}[1]{} 

\renewcommand{\tilde}{\widetilde}
\newcommand{\p}{\partial}
\renewcommand{\d}{\delta}

\newcounter{pcounter}

\newcommand{\DS}{\displaystyle}

\newcommand{\EE}{{\mathcal E}}

\newcommand{\OO}{{\mathcal O}}

\renewcommand{\SS}{{\mathcal S}}

\newcommand{\UU}{{\mathcal U}}
\newcommand{\VV}{{\mathcal V}}

\newcommand{\R}{\mathbb{R}}

\newcommand{\C}{\mathbb{C}}

\newcommand{\alp}{\alpha}
\newcommand{\bet}{\beta}
\newcommand{\gam}{\gamma}
\newcommand{\eps}{\epsilon}

\newcommand{\kap}{\kappa}

\renewcommand{\phi}{\varphi}

\newcommand{\sig}{\sigma}
\newcommand{\ang}{\theta}

\newcommand{\Ome}{\Omega}

\renewcommand{\Re}{\textrm{Re }}

\DeclareMathOperator{\sign}{sign}
\DeclareMathOperator*{\spt}{supp}

\newcommand{\ddt}{\, \frac{d}{dt}}
\renewcommand{\iint}{\int\!\!\!\!\int}

\def\XXint#1#2#3{{\setbox0=\hbox{$#1{#2#3}{\int}$}
\vcenter{\hbox{$#2#3$}}\kern-.5\wd0}}

\newcommand{\upref}[2]{\hspace{-0.8ex}\stackrel{\eqref{#1}}{#2}} 

\newcommand{\lupref}[2]{\hspace{0ex} \stackrel{\eqref{#1}}{#2}} 

\usepackage{color}
\definecolor{verylightblue}{rgb}{0.95, 0.95, 0.95}  
\definecolor{lightblue}{rgb}{0.7, 0.7, 1}

\definecolor{eqyellow}{rgb}{0.9375,0.8984,0.5469}
\definecolor{subeqyellow}{rgb}{1,0.9373,0.8353}

\definecolor{darkcyan}{rgb}{0.0, 0.45, 0.95} 

\definecolor{mygreen}{rgb}{0.3, 0.6, 0.3} 
\definecolor{verylightgreen}{rgb}{0.95, 0.95, 0.95} 
\definecolor{verydarkgreen}{rgb}{0, 0.5, 0}
\definecolor{darkgreen}{rgb}{0.1, 0.55, 0.1}  
\definecolor{mydarkgreen}{rgb}{0, 0.5, 0} 

\definecolor{mybrown}{rgb}{0.85, 0.4, 0.3}
\definecolor{verylightbrown}{rgb}{0.98, 0.72, 0.58}
\definecolor{verydarkbrown}{rgb}{0.44, 0.26, 0.26}

\definecolor{orange}{rgb}{1, 0.5, 0}

\definecolor{BurntOrange}{rgb}{1,0.356,0}

\definecolor{mydarkred}{rgb}{1,0.086,0.255}

\definecolor{RoseVYDP}{rgb}{0.84,0.086,0.255}

\definecolor{dgreen}{rgb}{0, 0.8, 0.5}     

\definecolor{CanaryBRT}{rgb}{1,0.76,0.26}

\definecolor{cyan}{rgb}{0, 1, 1}
\definecolor{verylightgray}{rgb}{0.95, 0.95, 0.95}
\definecolor{lightgray}{rgb}{0.8, 0.8, 0.8}
\definecolor{verylightred}{rgb}{1, 0.8, 0.78}
\definecolor{verylightyellow}{rgb}{0.99, 0.98, 0.5}

\newcommand{\ignore}[1]{{}}

\makeatletter
\newcommand{\VEC}[2][r]{
  \gdef\@VORNE{1}
  \left(\hskip-\arraycolsep%
    \begin{array}{#1}\vekSp@lten{#2}\end{array}%
  \hskip-\arraycolsep\right)}
\def\vekSp@lten#1{\xvekSp@lten#1;vekL@stLine;}
\def\vekL@stLine{vekL@stLine}
\def\xvekSp@lten#1;{\def\temp{#1}%
  \ifx\temp\vekL@stLine
  \else
    \ifnum\@VORNE=1\gdef\@VORNE{0}
    \else\@arraycr\fi%
    #1%
    \expandafter\xvekSp@lten
  \fi}
\makeatother

\definecolor{reviewdarkblue}{rgb}{0.0,0.2,1}
\definecolor{reviewlightblue}{rgb}{0.1,0.5,1}
\definecolor{reviewdeletedblue}{rgb}{0.5,0.8,1}
\definecolor{reviewdarkgreen}{rgb}{0.1,0.6,0.1}
\definecolor{reviewlightgreen}{rgb}{0.3,0.8,0.3}
\definecolor{reviewdeletedgreen}{rgb}{0.5,0.9,0.5}
\definecolor{reviewdarkred}{rgb}{0.66,0.05,0.05}
\definecolor{reviewlightred}{rgb}{0.74,0.3,0.4}
\definecolor{reviewdeletedred}{rgb}{0.95,0.82,0.86}

\DeclareMathOperator{\supp}{supp}

\let\OLDthebibliography\thebibliography
\renewcommand\thebibliography[1]{
  \OLDthebibliography{#1}
  \setlength{\parskip}{3pt}
  \setlength{\itemsep}{0pt plus 0.3ex}
}
\usepackage[
  a4paper,
  left=32mm,
  right=32mm,
  top=30mm,
  bottom=32mm
]{geometry}
\newtheorem{theorem}{Theorem}[section]

\newtheorem{proposition}[theorem]{Proposition}

\newtheorem{conjecture}[theorem]{Conjecture}
\newtheorem{remark}[theorem]{Remark}
\newtheorem{assumption}[theorem]{Assumption}

\setlist[enumerate,1]{
  label=\textrm{(\roman*)},
  leftmargin=2.2em,
  labelsep=0.5em,
  itemsep=0.2em,
  topsep=0.3em,
  parsep=0pt,
  partopsep=0pt
}
\setlist[enumerate,2]{
  label=\textrm{(\alph*)},
  leftmargin=2em,
  labelsep=0.5em,
  itemsep=0.1em,
  topsep=0.2em,
  parsep=0pt,
  partopsep=0pt
}

\newcommand{\LAM}{{\mathfrak s}}
\title[Vanishing-slip contact-line regimes]{Formal contact-line regimes in the vanishing-slip limit of the thin-film equation}

\author{Hans Knüpfer}
\address{Institut für Angewandte Mathematik \& Interdisciplinary Center for Scientific Computing, Universität Heidelberg,  
Im Neuenheimer Feld 205, 69120 Heidelberg, Germany}
\email{knuepfer@uni-heidelberg.de}
\author{Juan J.~L. Velázquez}
\address{Institut für Angewandte Mathematik, Rheinische Friedrich-Wilhelms-Universität Bonn,  
Endenicher Allee 60, 53115 Bonn, Germany}
\email{velazquez@iam.uni-bonn.de}

\begin{document}

\begin{abstract}
  We analyze the evolution of thin liquid droplets in the lubrication
  approximation with different slip conditions at the liquid--solid
  interface. Motivated by the classical no-slip paradox which states that the
  Navier–Stokes equations with a no-slip boundary condition require unphysical
  infinite dissipation during droplet spreading, we focus on the limit of
  vanishing slip.  Using matched asymptotic expansions, we formally identify two
  distinct limiting regimes in the vanishing-slip limit and suggest a third,
  conjectural regime, each of them corresponding to a different scaling of the
  microscopic contact angle as the regularization parameter vanishes.  These
  findings suggest the no-slip thin-film equation may admit several distinct
  limiting evolutions when it is interpreted through vanishing-slip
  regularizations. Even though the large apparent contact angles in some of
  these solutions seem incompatible with the lubrication approximation, the
  rescaling analysis indicates that the underlying physical variables remain
  consistent with the assumptions for the lubrication approximation.
\end{abstract}
 
\renewcommand{\em}[1]{\textit{#1}}

\maketitle
\tableofcontents

\section{Introduction and main results}


\subsection{Introduction}

We consider the evolution of a liquid droplet on a solid substrate. In the
classical fluid dynamical description, the evolution in the fluid is described
by the Navier--Stokes equations, while a no slip condition is assumed at the
liquid--solid interface. It has been observed that such a no--slip condition
allows for droplet spreading only under the seemingly unphysical condition of
infinite dissipation of energy, e.g. \cite{HuhScriven-1971}. This apparent
\textit{no slip paradox} has led to some controversies about possible
resolutions of this paradox. We want to discuss some related issues. As one way
to avoid the no--slip paradox different relaxed slip conditions have been
proposed where the tangential velocity, relative to the solid, at the interface
is a function of the tangential shear stress. In particular, we consider the
evolution of a thin capillary driven droplet in the so--called lubrication
approximation regime. Assuming a no slip condition, in this regime the evolution
of the droplet can be described by the thin--film equation
\begin{align} \label{tfe-noslip} 
    h_t +(h^3 h_{xxx})_x \ = \ 0 \qquad\qquad %
  \text{in $\Ome \ := \ \{h>0\}$.}
\end{align}
Here, $h$ is the droplet height and the continuity assumption $h = 0$ is assumed
at $\p \Ome$.  Assuming more general slip conditions, the lubrication
approximation leads to the family of equations
\begin{align} \label{tfe-slip} 
  h_t +(m_\eps(h) h_{xxx})_x \ = \ 0\qquad\qquad %
  \text{in $\Ome \ := \ \{h>0\}$.}
\end{align}
Here and for later use, we write
\begin{align}
  m_\eps(h)
  \ = \
  h^3+\eps^{3-n}h^n,
\end{align}
with $m_0(h)=h^3$ in the no--slip case.  The mobility exponent $n \in (0,3)$
models the slip condition at the liquid--solid interface: $n = 2$ represents
Navier--slip where a linear slip/stress relationship is assumed, while $n = 1$
corresponds to Greenspan's slip condition \cite{Greenspan-1978}. The
regularization parameter $\eps > 0$  is assumed to be small and for
  $n =2$ corresponds to the  \textit{slip length} (see Section
\ref{LubApp}). We recover equation \eqref{tfe-noslip} in the limit $\eps \to
0$. Equation \eqref{tfe-slip} has to be complemented by a boundary condition at
$\p \Ome$. We assume the constant angle condition
\begin{align} \label{tfe-bc} %
  h_\nu \ = \ -\ang_\eps \qquad\qquad\qquad %
    \text{at $\p \Ome \ = \ \p\{h>0\}$,}
\end{align}
where $\nu$ is the exterior unit normal to the positivity set. We note
that the boundary condition \eqref{tfe-bc} for \eqref{tfe-noslip} leads to an
ill--posed problem \cite{EggersFontelos-2025}. Throughout the paper, for both
equations \eqref{tfe-noslip} and \eqref{tfe-slip} we also use the boundary
condition
\begin{align}\label{no-massflux}
  \lim_{\Ome\ni x\to x_0}
  m_\eps(h(x,t))h_{xxx}(x,t)
  \ = \
  0
  \qquad
  \text{for every $x_0\in\p\Ome(t)$}.
\end{align}
This condition ensures that there is no loss of mass across the contact points
$\p \Ome$. For further references we refer e.g. to
\cite{BDQ-Book,HuhScriven-1971,Hocking-1983,EggersStone-2004}, for nonlinear
slip see e.g. \cite{ThompsonTroian-1997}. We note that other regularizations of
the no--slip paradox exist which are not considered here, see
e.g. \cite{OronDavisBankoff-1997,WitelskiBernoff-1999,OronBankoff-2001,BowenKing-2001}. For
existence of solutions for \eqref{tfe-slip} we refer e.g. to
\cite{BertozziPugh-1994,BerettaBertschDalpasso-1995,BertschDalpassoGarckeGruen-1998,DalpassoGarckeGruen-1998,Gruen-2004-1,GiacomelliKnuepferOtto-2008,GnannWisse-2022,GnannWisse-2025}.
Receding contact-line profiles with the leading behaviour \eqref{def-typeb} were
already considered in \cite{BowenKing-2001} by Bowen and King.  Here, we
organize the formal vanishing-slip asymptotics according to the scaling of the
microscopic contact angle. In particular, we identify the critical scaling
$\ang_\eps\approx\gam|\ln\eps|^{\frac13}$ associated with order-one recession
and discuss the cancellation between divergent viscous dissipation and
contact-energy release under additional asymptotic assumptions. The faster
concentration regime remains conjectural. The literature is discussed more in
detail in the next section.

\medskip

More in particular, we study solutions of \eqref{tfe-noslip} which appear as a
limit of solutions for \eqref{tfe-slip} as $\eps \to 0$. Our arguments are
formal in the sense that they are based on the use of matched asymptotic
expansions.  For simplicity, we consider solutions with an isolated contact
point $s(t)$ at the left boundary of the support, i.e
$\Ome \cap B_{\d}(s(t)) = (s(t), s(t) + \d)$ for some $\d > 0$. We formally
identify the following three limiting regimes:

\medskip

\textbf{Regime (a):} \label{Typea} %
  The support is constant, $s(t)= s_{0}$. Furthermore, 
\begin{align} \label{S1E4} %
  h(x,t) \ \approx \ \gam(t) \left(x-s_{0}\right) \qquad\qquad %
  \text{ as $x \to  s_0^+$}
\end{align}%
for some smooth function $\gam(t)$ (see the paragraph on notation below for the
meaning of the symbol ``$\approx$''). Theorem \ref{thm-nomove} gives a
  sufficient regularity criterion ensuring that the support is constant. A
  contact angle $\ang_\eps$ of order one hence corresponds formally to type
  (a) solutions. We believe that type (a) solutions also appear in the remaining
  case where $1 \ll \ang_\eps \ll |\ln \eps|^{\frac 13}$ but this case is not
  treated in this paper. 

\medskip

\textbf{Regime (b):} \label{Typeb} %
The formal local expansion is compatible with a prescribed smooth function
$s(t)$ satisfying $\dot s(t)>0$, corresponding to a shrinking droplet. The
solutions have asymptotic behaviour %
  \begin{align}
    h(x,t) \ %
    \approx \ (3 \dot s(t))^{\frac 13}\left(x-s(t) \right) \Big(\ln \frac{1}{
    x-s(t)} \Big)^{\frac 13} \qquad\qquad %
    \text{ as $x \to  s(t)^+$}  \label{def-typeb}
  \end{align}
  with notation $\dot s = \frac d{dt} s$. We remark that the possibility of
  having travelling wave solutions with receding support, as in the type (b)
  regime was already suggested in \cite{BoattoEtal-1993} and obtained by
  matched asymptotics in \cite{BowenKing-2001}. A contact angle in the intermediate regime
    $\ang _{\eps } \approx \gam |\ln \eps|^{\frac 13}$ formally leads to type
    (b) solutions. The contact point velocity in the intermediate regime is given by
    \begin{align} \label{dots-formula} %
      \dot s \ %
      = \   \lim_{\eps \to 0} \frac {\ang_\eps^3}{3|\ln \eps|} \ %
      = \ \frac {\gam^3}3 \ > \ 0, \qquad %
      \text{where $\ang_\eps \approx \gam |\ln \eps|^{\frac 13}$.}
    \end{align} 
    A rigorous construction
    determining whether distinct prescribed contact-line paths can arise from
    the same initial data is left open.

\medskip

\textbf{Regime (c):} \label{Typec} %
  The third regime is a conjectural fast-time measure-valued limit. The
  interface $s$ is also monotonically increasing and moves on a shorter time
  scale $\tau$ than $t,$ that will be denoted as $\tau $ and the motion as
  $s\left(\tau \right)$. The mass accumulates in the moving interface
  $s(\tau)$. Therefore with $s' := \frac{d}{d\tau}s$ %
  \begin{align}
  h(x,\tau) =m(\tau) \delta_{s(\tau)}(x) +h_{0}(x) \chi_{\{ x>s\left(\tau
  \right)\} }  \ \  %
  \text{with} \ m'(\tau) = h_{0}(s(\tau)) s'(\tau)
\label{def-typec}
\end{align}
For these solutions the interface moves so fast that the profile $h_{0}$ does
not have time to change. In this case all the mass of $%
h_{0}$ that is reduced due to the reduction of the support accumulates at the
interface boundary. A contact angle of size
  $\ang _{\eps } \gg |\ln \eps|^{\frac 13}$ formally leads to type (c)
  solutions. This ansatz preserves the total mass, but it is not claimed to
  satisfy the constitutive no-slip thin-film equation, see the upcoming paper
  \cite{paper2}.

\medskip

The above formal regimes are expected to arise as limits of classical
  solutions of \eqref{tfe-slip} under different scalings of the microscopic
  contact angle as $\eps \to 0$. We also note that the limits above might only
exist for a subsequence $\eps_k \to 0$ where the limiting angle has a
well--defined behaviour. Assuming that \eqref{tfe-slip} is a valid physical
regularization of \eqref{tfe-noslip} this suggests that the three types of
solutions (a), (b) and (c) are physically feasible as well. We give formal
  local expansions for types (a) and (b), and a heuristic description of type
  (c) for equation \eqref{tfe-noslip}.  We  also give some arguments
how these solutions appear as a limit for solutions of \eqref{tfe-slip}. A
rigorous proof for the existence of type (a) solutions will be given in the
upcoming paper \cite{GKV-preprint}. We also note that the solutions of types
(a), (b) and (c) are consistent with the lubrication approximation.
We refer to Remark \ref{rem-consist} for more details.

\medskip

The mathematical status of the three regimes is different. Theorem
\ref{thm-nomove} is a rigorous pinning result under suitable growth and
time-derivative assumptions. The constructions of types (a) and (b) are formal
matched-asymptotic arguments. On the other hand, the solutions of type (c) are
conjectured measure-valued solutions, valid in a fast-time limit.
\begin{remark}[Time scales] \label{rem-time} The evolution involves a bulk time
  scale and a contact-point time scale. In the nondimensionalized variables, the
  bulk profile changes on time scales of order one. The contact-point time scale
  depends on the microscopic contact angle: it is long in Regime (a), of order
  one in Regime (b), and short relative to the bulk evolution in Regime (c).
\end{remark}

\begin{remark}[Dissipation of energy]
  Equations \eqref{tfe-noslip} and \eqref{tfe-slip} describe an evolution
  governed by the dissipation of capillary energy. The energy--dissipation
  formula for \eqref{tfe-slip} has the form%
\begin{align}\label{diss-intro}
  \frac d{dt}
  \frac12
  \int_{\Ome(t)}
  \bigl(h_x^2+\ang_\eps^2\bigr)
  \ dx
  \ = \
  -\int_{\Ome(t)} m_\eps(h) h_{xxx}^2  \ dx, 
    \end{align}
    assuming that the solution is sufficiently regular. The two terms on the
      left--hand side represent the relative liquid--gas and solid--liquid
      interfacial energies, while the right--hand side is the viscous
      dissipation. For a type (b) solution of \eqref{tfe-noslip}, the local
      dissipation diverges: for every $\d>0$,
\begin{align}\label{S1E6}
  \int_{s(t)}^{s(t)+\d} h^3h_{xxx}^2 \ dx
  \ = \
  \infty.
\end{align}
This does not by itself exclude type (b) solutions as vanishing-slip limits.
Indeed, when $\ang_\eps\approx\gam|\ln\eps|^{\frac13}$, both the viscous
dissipation and the energy released by the shrinking wetted interval diverge as
a power of $|\ln\eps|$. Proposition~\ref{prp-typebhoDISS} shows, under some
additional assumptions, that these divergent contributions cancel at leading
order and have a finite difference.
\end{remark}


\medskip

\textbf{Acknowledgement.} This project was initiated during a meeting at the
Lorentz Center in Leiden in June 2023, organized by L. Giacomelli, M. Gnann,
J. Hulshof, C. Lienstromberg and S. Sonner.  In addition to discussions with the
organizers there were many discussions with several participants such as
M. Bertsch, J. Eggers, G. Grün and J.R. King. The authors also gratefully
acknowledge the support by the Deutsche Forschungsgemeinschaft (DFG) through the
collaborative research centre ``The mathematics of emerging effects'' (CRC 1720
-539309657) and Germany's Excellence Strategy EXC2047/1-390685813. The funders
had no role in study design, analysis, decision to publish, or preparation of
the manuscript.  HK was partially supported by the German Research Foundation
(DFG) by the project \#392124319 and under Germany's Excellence Strategy –
EXC-2181/1 – 390900948.

\medskip

\textbf{Use of artificial intelligence.} The research and the original
manuscript were developed independently by the authors. ChatGPT (OpenAI) was
subsequently used to assist in checking the manuscript for possible errors and
inconsistencies and in improving the final version of the text. The authors
assume responsibility for all content.

\medskip

\textbf{Competing Interests.} The authors declare that they have no competing
interests.

\subsection{Related results and literature}

Much work has been  devoted to the related family of equations
\begin{align}
  h_t +(h^n h_{xxx})_x\ = \ 0\qquad \text{in $\Ome := \{h>0\}$}  \label{tfe-n}
\end{align}%
for $n > 0$. In fact, for $n \in (0,3)$ solvability of this equation can be
expected to be the same as for \eqref{tfe-slip} since only the behaviour of the
leading term $h^n h_{xxx}$ as $h\to 0$ matters. For $n \in (0,3)$ two different
types of boundary conditions have been mostly considered: In the so--called
\textit{complete wetting regime} a zero contact angle condition is assumed; in
the partial wetting regime a non--zero contact angle is assumed. The contact
angle can be fixed and determined by Young's Law (as considered in this paper),
or variable.

  \medskip

  Existence of weak solutions of \eqref{tfe-n} in the complete wetting regime
  has been shown in
  \cite{BernisFriedman-1990,BertozziPugh-1994,BerettaBertschDalpasso-1995,Gruen-2004-1}
  for $n\in (0,3)$. Furthermore, the support of solutions with zero contact
  angle cannot shrink for $n\geq \frac{5}{2}$
  \cite[eq. (0.5)]{BerettaBertschDalpasso-1995}. For $n\geq \frac 32$, also the
  positivity set cannot shrink. Furthermore, for $n\geq 4$ the support of
  solutions remains constant. For existence of classical solutions we refer to
  e.g. \cite{GiacomelliKnuepferOtto-2008,GiacomelliKnuepfer-2010,GiacomelliGnannKnuepferOtto-2014,Gnann-2015,GnannPetrache-2018,GessGnann-2020}. Existence
  of weak solutions for stochastic thin--film equations has been considered
  e.g. in \cite{FischerGruen-2018,GessGnann-2020,DareiotisEtal-2021}. Relaxation
  to the stationary solution has been considered in
  e.g. \cite{CarrilloToscani-2002,Esselborn-2016}. Existence of weak solutions
  in the partial wetting regime has been established by Otto in \cite{Otto-1998}
  for $n=1$ and by Bertsch, Giacomelli and Karali in
  \cite{BertschGiacomelliKarali-2005} for $n\in (0,3)$. Finite speed of
  propagation is e.g.  shown in
  \cite{Bernis-1996-1,Bernis-1996-2,HulshofShishkov-1998,Gruen-2003}. For
    suitable initial data, short-time waiting phenomena occur, see e.g.
  \cite{DalpassoGiacomelliGruen-2001,BloweyKingLangdon-2007,Shishkov2007,Fischer-2014}. Existence
  of classical solutions with stationary support for $n = 3$ is work in
  preparation \cite{GKV-preprint}. At the contact point, these solutions have
  the behaviour $h(x,t)\approx b(t)(x-s(t))$, where $s(t)$ is the contact point.
  The long--time asymptotics of complete wetting solutions has been formally
  studied by Giacomelli, Gnann and Peschka in
  \cite{GiacomelliGnannPeschka-2023}.  Travelling wave solutions have been
  systematically investigated by Boatto, Kadanoff and Olla
  \cite{BoattoEtal-1993} and by King and Taranets in \cite{KingTaranets-2013}.

  \medskip

  The relation between the speed of the contact point for the model
  \eqref{tfe-slip} and the macroscopic or apparent contact angle is known as
  Cox--Voinov law, we refer to the works by Voinov, Tanner, Cox and Hocking
  \cite{Voinov-1977,Tanner-1979, Cox-1986, Hocking-1992}, see also
  \cite{BonnEggersEtal-2009}. In \cite{Eggers-2004}, Eggers has derived higher
  order corrections to this Law. Existence of travelling wave solutions which
  satisfy this law has been shown by Giacomelli, Gnann and Otto in
  \cite{GiacomelliGnannOtto-2016} in the zero contact angle case and by Gnann
  and Wisse \cite{GnannWisse-2022} in the case of non--zero contact
  angle. Delgadino and Mellet show that solutions of the thin--film
  equation in the regime of small slip converge to solutions of a
  quasistatic evolution problem \cite{DelgadinoMellet-2021}.

\medskip

The quantitative behaviour near the contact line can be best seen by the ODE 
\begin{align} \label{main-ODE} %
  (h^2+ \eps^{3-n} h^{n-1})h_{\xi\xi\xi}  \ %
  = \ \pm 1\qquad \text{for $\xi \in  (0,\infty )$},
\end{align}%
which emerges e.g. from a travelling wave solution ansatz after one
  integration and normalization of the nonzero wave speed. The positive sign on
the right--hand side means a shrinking droplet, while the negative sign models
an expanding droplet. This equation for $n=3$ and $\eps =0$ has been explored
e.g. in \cite[eqs. (4.3.9),(4.3.23)]{BenderOrszag-Book}. We note that
\eqref{main-ODE} for values $n<3$ has been investigated in \cite[$r = 1$
corresponds to $m = 3$]{BoattoEtal-1993}. The equation for $n=3$ has also been
investigated in \cite{DuffyWilson-1997} in the context of Tanner's Law, relating
radius of the droplet with the time.

\medskip

Eggers and Fontelos \cite{EggersFontelos-2025} have recently considered the case
of a pinned contact line both analytically and numerically. They show that, for
$n \neq 3$, the contact angle changes like a power law, while the dynamics
become nonlocal in the critical case $n = 3$. With this result, the authors
conjecture that the contact line problem with prescribed contact angle becomes
ill--posed in the case $n = 3$.

\medskip

In addition to assuming a slip at the interface, other physical mechanisms have
been proposed as well to resolve the no--slip paradox, for reviews we refer
e.g. to \cite%
{BDQ-Book,OronDavisBankoff-1997,BonnEggersEtal-2009,SnoeijerAndreotti-2013,Shikhmurzaev-2020,DurastantiGiacomelli-2024}.
For the situation when there is evaporation of the liquid at the contact line,
we refer to \cite{RednikovColinet-2020}.  The results in that paper suggest the
existence of another class of solutions having non--trivial fluxes. Actually,
the presence of fluxes at the contact line would imply
$h\approx (x-s(t))^{\frac 34}$ near the contact point.  For computational
approaches to spreading of thin--films we refer e.g. to
\cite{BarrettBloweyGarcke-1998,MuenchWagnerWitelski-2025,Peschka-2015,Peschka-2018,PeschkaHeltai-2022}.

\subsection{Notation.}

We will use the notation for asymptotic formulas, i.e. we write $%
f(x) \approx g(x)$ as $x\to x_{0}\in [ -\infty ,\infty] $ to denote that
$\lim_{x\to x_{0}}\frac{%
  f(x) }{g(x) } =1$ and $f(x) \ll g(x)$ as $x\to x_{0}\in [ -\infty ,\infty] $
if $%
\lim_{x\to x_{0}}\frac{f(x) }{g(x)} =0$. We use the notation $\OO(\cdot)$
  and $o(\cdot)$ when quantitative information on the remainder is relevant. 
We write $f(x)\sim g(x)$ if there are constants $c,C>0$ such that
$c < \frac{f(x)}{g(x)} < C$ for $x$ sufficiently close to $x_0$. Notice that in
this definition no assumptions are made about the behaviour of the function or
its derivatives as $(x,t)$ approaches $\partial \left\{ h>0\right\}$.

\section{Physical background \label{LubApp}}

\begin{figure}
  \centering
  \includegraphics[width=6cm]{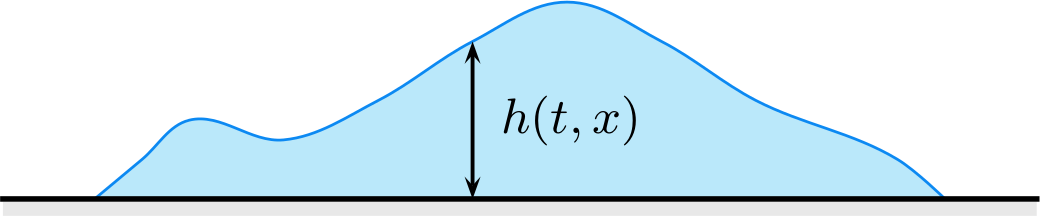}
  \caption{Thickness profile $h(t,x)$ of an evolving droplet.}
  \label{fig-droplet}
\end{figure}

\subsection{Lubrication approximation} \label{ss-lubapp} %

We recall the lubrication approximation that is used to derive the equation
\eqref{tfe-slip} for $n = 2$ taking as starting point the free boundary problem
for the Stokes or Navier-Stokes equations describing the evolution of a droplet
on a solid substrate (see Fig. \ref{fig-droplet}). Since in the derivation of
the thin film equation, the inertial terms give a negligible contribution, we
will consider only the free boundary problem for the Stokes problem. For
simplicity, we consider the two-dimensional case. Since this derivation is
standard we will just indicate the main steps, see also
\cite{GiacomelliOtto-2003,GuentherProkert-2008,MatiocProkert-2012,KnuepferMasmoudi-2013,KnuepferMasmoudi-2015}. We
assume that the geometry of the fluid domain is
\begin{align}
  \UU(T) \ := \ \big \{ X \in (S_1,S_2), \ 0 < Y < H(X,T)  \big \},
\end{align}
where $H$ is the droplet height, satisfying $\spt H = [S_1(T),S_2(T)]$ also
including the case $S_2(T) = \infty$. We write
$\p_1 \UU := \OL{\p \UU \cap \{ Y > 0 \}}$ for the liquid--gas interface and
$\p_0 \UU := \p \UU \cap \{ Y = 0 \}$ for the solid--liquid interface. The
evolution is given by the free boundary Stokes flow
\begin{align+}
  \mu \Delta V \ %
  &= \nabla P, %
  &\nabla \cdot V %
  &= \ 0\ \ %
  &&\text{in\ \ }\UU(t)   \label{S2E1} \\
  V_1 \  %
  &= B_\bet(H)\p_YV_1,
  &V_2  \ %
  &= \ 0
  &&\text{on $\p_0\UU(t)$}, \label{S2E3} \\
  2 \mu \SS N \ %
  &= (P - \gam_{LG} \kap) N  %
  &&&&\text{on $\p_1 \UU(t)$},  \label{S2E5} \\
  V_N \ %
  &= \ V\cdot N 
  &&&&\text{on $\p_1 \UU(t)$},  \label{S2E5b} \\
  H_\nu \ %
  &= \ -\tan \Theta \ %
  &&&&\text{at $\p_0 \UU \cap \p_1 \UU$}, \label{S2E6}
\end{align+}
where $\nu$ is the outer normal to the support boundary. Here,
  $B_\bet(H) \ = \ \frac13\bet^{3-n}H^{n-2}$ where $\bet$ is a microscopic
  length parameter and $B_\bet(H)$ is the effective height-dependent slip
  length. For $n=2$, this reduces, up to the normalization factor $\frac 13$, to
  the classical Navier slip condition. For general $n<3$, it is a generalized
  Navier condition chosen to generate the regularized mobility considered
  below. where $\SS := \frac 12 (\nabla V + \nabla^t V)$. Here, $V(X,T)$ is the
fluid velocity, $P(X,T)$ is the pressure, $\kap$ is the curvature of the
interface, $N$ its normal and $V_N$ is the normal velocity of the
interface. Furthermore, $\mu$ is the viscosity, $\gam_{LG}$ the surface tension
of the liquid-gas interface and $\beta$ is the microscopic length parameter
entering the slip law. On the liquid--gas interface, we assume that the viscous
stress is balanced by the surface tension and the interface moves with the
fluid. At the liquid--solid interface, we assume a Navier slip boundary
condition and a non--penetration condition. The microscopic length parameter
$\bet$ is assumed to be small and noting that in the limit $\bet \to 0$ we
obtain the classical no-slip boundary condition. The system of equations
\eqref{S2E1}--\eqref{S2E6} for some given contact angle law is formally
complete. Well--posedness results for related free-boundary Stokes systems have
been obtained, for example, in \cite{GuoTice-2018,LeoniTice-2023}, for static or
corresponding dynamic contact-angle conditions. The study of type (c)
  solutions requires additional information about the Landau--Levich problem and
  will also be considered in future work.

\medskip

We assume the contact angle $\Theta$ to be fixed and determined by Young's Law
\begin{align} \label{quot} %
  \cos \Theta \ %
  = \ \frac{\gam_{SG}-\gam_{SL}}{\gam_{LG}} \ \in \ (-1,1), %
\end{align}
which is the equilibrium condition for the interfacial energy at the contact
point. Here, $\gam_{SL}$, $\gam_{SG}$ and $\gam_{LG}$ are the surface tensions
between solid, liquid and gas, respectively. The angle $\Theta$ is also called
microscopic contact angle which might differ from the experimentally observed
macroscopic contact angle. Solutions dissipate the relative interfacial
energy,
\begin{align} \label{energy-init} %
  \EE[H] \ %
  = \ \gam_{LG} \int_{\p_0 \UU} \big( \sqrt{1+H_{X}^2} -\cos \Theta \big) \ dX.
\end{align}
Here, $\p_0\UU$ is identified with its horizontal projection onto the wetted
  solid interval, and the constant energy of the dry solid has been
  subtracted. In this paper, we mostly consider the small-angle partial-wetting
  regime, where the right-hand side of \eqref{quot} is close to one, and hence
  $\Theta \ll 1$.

\medskip

We sketch the formal derivation of the thin film equation \eqref{tfe-slip} using
the lubrication approximation starting from \eqref {S2E1}--\eqref{S2E6}. We
first introduce a typical vertical length scale $\LAM_y$ and a typical
horizontal length scale $\LAM_x$: We choose $\LAM_y$ as the maximal droplet
thickness, assuming that this number is finite. Furthermore, we define $\LAM_x$
as a typical length scale where the change of the height from the contact point
is of the same order as $\LAM_y$.

\medskip

Capillary and viscous forces balance at leading order, while inertial effects
are negligible which justifies that we choose the Stokes flow as initial model.
The derivation is furthermore based on the following scaling assumptions:
\begin{assumption}[Lubrication scaling] \text{} %
\label{ass-lubrication}
\begin{enumerate}
\item The horizontal length scale is much larger than the vertical length
scale:
\begin{align}
  \d
  \ := \
  \frac{\LAM_y}{\LAM_x}
  \ \ll \
  1.
\end{align}
The profile satisfies the corresponding long-wave scaling
$|H_{XX}|\lesssim\LAM_y/\LAM_x^2$.
\item The droplet thickness is much larger than the microscopic slip length
\begin{align}
  \eps
  \ := \
  \frac{\bet}{\LAM_y}
  \ \ll \
  1.
\end{align}
\end{enumerate}
\end{assumption}
The parameters may depend on $\eps$; in particular, we allow $\d=\d_\eps$. We
choose typical length scales $\LAM_p$ and $\LAM_t$ for pressure and time
by
\begin{align}\label{res-2} \LAM_p \ := \ \frac{\gam_{LG}\LAM_y}{\LAM_x^2}
  \ = \ \frac{\gam_{LG}\d}{\LAM_x}, \qquad\qquad %
  \LAM_t \ := \ \frac{\mu\LAM_x^4}{\gam_{LG}\LAM_y^3} \ = \
  \frac{\mu\LAM_x}{\gam_{LG}\d^3}.
\end{align}
We introduce the rescaled (non--dimensional) variables
\begin{align} \label{res-var} %
  \big( x, y, h, t, v_1, v_2, p, \ang \big) \  %
  := \ \Big( \frac{X}{\LAM_x}, \frac{Y}{\LAM_y}, \frac{H}{\LAM_y}, \frac{T}{\LAM_t}, \ %
   \frac{V_1 \LAM_t}{\LAM_x}, \frac{V_2 \LAM_t}{\LAM_y}, \ \frac{P}{\LAM_p}, \frac{\Theta }{\delta } \Big) %
\end{align}
and the nondimensionalised slip coefficient by
\begin{align}
  b_\eps(h)
  \ := \
  \frac{B_\bet(\LAM_yh)}{\LAM_y}
  \ = \
  \frac13\eps^{3-n}h^{n-2}.
\end{align}
The rescaled fluid domain is $\VV(t):=\{(x,y):x\in\Ome(t),\ 0<y<h(x,t)\}$. 
For $n\neq2$, the phenomenological height-dependent Navier coefficient is chosen
so that the lubrication limit has mobility $h^3+\eps^{3-n}h^n$. Only the case
$n=2$ corresponds to a constant Navier slip length.  By our choice of variables,
solutions of the limit equation change in order one time scales
(cf. Rem. \ref{rem-time}).  At leading order in $\d\to0$, the rescaled Stokes
equations give
\begin{align}
  v_{1,yy} \ %
  &= \ p_x, \ 
  &p_y \ %
  &= \ 0
  &&\text{in $\VV$.} \\
  v_{1,y} \ %
  &= \ 0,
  &p \ %
  &= \ -h_{xx}
  &&\text{on $\p_1 \VV$}, \\
  v_1 \ %
  &= \ b_\eps(h)v_{1,y},
    &v_2 \ %
  &= \  0
  &&\text{on $\p_0 \VV$}.
\end{align}
Hence, $v_{1,yy} = -h_{xxx}$ in $\VV$. Integrating the equation for $v_1$ and
using $v_{1,y}(h)=0$, we obtain $v_{1,y}(y)$ $=$ $(h-y)h_{xxx}$ and hence
\begin{align}
  v_1(y)
  \ &= \ \Big(
  hy-\frac12y^2+b_\eps(h)h \Big)h_{xxx}.
\end{align}
Using incompressibility, the non--penetration condition at $y=0$, and the
kinematic condition at $y=h$, we obtain $h_t+q_x = 0$, where the horizontal flux
$q$ is given by
\begin{align} \label{fl-2}
  q  \ %
  := \  \int_0^h v_1(y)  \ dy  \ %
  &= \  \Big(  \frac13h^3+b_\eps(h)h^2  \Big)h_{xxx}  \ %
    = \  \frac13  ( h^3+\eps^{3-n}h^n )h_{xxx}.
\end{align}
Inserting \eqref{fl-2} into the conservation law and rescaling time by a
constant  we obtain \eqref{tfe-slip}.

\medskip

It might seem inconsistent that we allow for large contact angles in the
 lubrication approximation. However, this contradiction is resolved by the
 difference between physical and rescaled variables as indicated in Section
 \ref{LubApp}.

 \begin{remark}[Physical and rescaled contact angles]
\label{rem-consist}
The lubrication approximation requires the physical contact angle
$\Theta_\eps$ to be small. Under the anisotropic scaling we have
\begin{align}
  \Theta_\eps
  \ = \
  \d_\eps\ang_\eps.
\end{align}
Thus, the lubrication approximation remains consistent provided
$\d_\eps\ang_\eps\to0$, even if the rescaled contact angle $\ang_\eps$ diverges
as $\eps\to0$. In particular, all regimes (a),(b) and (c) can be consistent with
the lubrication approximation. Regime (b) is compatible with the lubrication
  approximation if $\d_\eps|\ln \eps|^{\frac 13} \to 0$.
\end{remark}
Since the contact angle is determined by Young's Law we can also phrase the
regimes (a),(b) and (c) from the introduction in terms of the interfacial
energies.
\begin{remark}[Interfacial energies and contact angle]
  For $\Theta \ll 1$  we have
  $\frac 12 \Theta^2 \approx 1 - \cos \Theta$ and hence by Young's law
  \begin{align}
      \frac 1{\gam_{LG}} (\gam_{LG} + \gam_{SL} - \gam_{SG})  \ %
      \lupref{quot}\approx  \ \frac {\Theta^2}{2} \ %
      = \ \frac 12 \ang_\eps^2 \d_\eps^2 \ %
      \to \ 0^+ \ \qquad %
      \text{as\ $\eps \to 0$.}  \label{S4E2}
  \end{align}
  In particular, the interfacial energies $\gam_{SG}$, $\gam_{SL}$, $\gam_{LG}$
  need to be adjusted in a very careful way to obtain the dynamics of
  the three different regimes.
\end{remark}

\subsection{Dissipation of energy  \label{sec-regime}}

Solutions of the initial Stokes system dissipate the total interfacial energy
\eqref{energy-init}. With a small angle expansion of the energy
\eqref{energy-init} and in terms of the rescaled variables \eqref{res-var} we get
\begin{align}
  \frac{\EE[H]}{\gam_{LG} \LAM_x \d^2} \  %
  \lupref{energy-init} = \ \frac 1{\LAM_x\d^2} \int_{\p_0 \UU} \sqrt{%
  1+H_{X}^{2}} -\cos \Theta \ dX  \ %
  \approx  \
  \frac 12  \int_{\Ome} \big(h_x^2+ \ang^2 \big) \ dx.
\end{align}
where the right hand side is the thin--film energy.  The thin film equation
\eqref{tfe-slip} satisfies an energy--dissipation identity with respect to
quantity on the right--hand side above.  A definition for weak solutions in the
partial wetting regime based on the dissipation of energy has been given by
Giacomelli and Shishkov in \cite[Def. 3.1]{GiacomelliShishkov-2005}. Here, we
recall this identity for classical solutions that are smooth up to the boundary.
\begin{proposition}[Energy--dissipation identity]
\label{prp-energy-dissipation}
Let $n\in (0,3]$ and $\eps>0$ and $\Ome(t)=(s_-(t),s_+(t))$. Suppose that $h$ is
a solution of \eqref{tfe-slip}, \eqref{tfe-bc} and
\eqref{no-massflux} which is smooth up to the boundary. Then
\begin{align}\label{T5E4}
  \frac d{dt}
  \frac12
  \int_{\Ome(t)}
  \bigl(h_x^2+\ang_\eps^2\bigr)
  \ dx
  \ = \
  -\int_{\Ome(t)}
  m_\eps(h)h_{xxx}^2
  \ dx.
\end{align}
\end{proposition}
\begin{proof}
  Differentiating the moving-domain energy and integrating by parts gives
\begin{align}
  \frac d{dt}   \frac 12 \int_{\Ome(t)}
  \bigl(h_x^2+\ang_\eps^2\bigr)
  \ dx
  \ = \
  - \int_{s_-}^{s_+}h_{xx}h_{t}
  \ dx
  +
  \Big[h_xh_t + \frac12   \dot s\left(h_x^2+\ang_\eps^2\right)  \Big]_{s_-}^{s_+}.
\end{align}
Differentiating $h(s_\pm(t),t)=0$ gives the kinematic identity
$h_t(s_\pm(t),t) + \dot s_\pm(t)h_x(s_\pm(t),t)$ $=$ $0$.  Together with the
boundary condition $|h_x|=\ang_\eps$ the boundary terms vanish.  Another
integration by parts and using \eqref{no-massflux} yields
\begin{align}
  \frac d{dt} \frac 12 \int_{\Ome(t)}
  \bigl(h_x^2+\ang_\eps^2\bigr)
  \ dx 
  \ &= \
  -\int_{s_-}^{s_+}h_{xx}h_t
  \ dx
  = \ -\int_{s_-}^{s_+}m_\eps(h)h_{xxx}^2 \ dx  + [h_{xx}m_\varepsilon(h)h_{xxx}]_{s_-}^{s_+}.
\end{align}
The boundary term vanishes by the smoothness assumption on the solution.
\end{proof}
The limiting type (b) profiles are not smooth up to the contact line and are
not covered by this proposition. For these profiles, the boundary contribution
associated with contact-line motion and the dissipation integral diverge
separately. Their leading divergences cancel when evaluated with a common
cutoff, as shown in Proposition~\ref{prp-typebhoDISS} below.

\medskip

We next discuss the dissipation of energy for type (a) solutions.
\begin{remark}[Dissipation of energy for type (a) solutions] %
  Suppose that $\Ome(t)$ is a bounded interval. The formal picture suggests
    two stages of evolution:
  \begin{enumerate}
  \item In the first stage, where $t=\OO(1)$, the droplet profile relaxes
      while the contact points remain approximately fixed.
    \item During the second stage where $t = \OO(|\ln \eps|)$, the leading
      dissipation is concentrated near the contact line.  The partial-wetting
      analysis in Section~\ref{ss-TannerCV} (see \eqref{H1-form-before} and
      \eqref{S7E4}) gives a dissipation rate of order $|\ln\eps|^{-1}$ and
      contact-line velocities of order $|\ln\eps|^{-1}$. Thus, $\OO(1)$ changes
      of $|\Ome(t)|$ occur on time scales of order $|\ln\eps|$. 
      \end{enumerate}
    \end{remark}
    Type (b) solutions are obtained in the regime
    $\ang_\eps \approx \gam |\ln \eps|^{\frac 13} \gg 1$. For type (b)
    solutions the system relaxes energy by reducing the size of the
    solid--liquid interface. The energy gained in this way is divergent as a
    power of a logarithm compensated with infinitely large dissipation. The
      formal arguments in Section~3.3 show that the local asymptotic expansion
      is compatible with a $C^1$ interface.

      \medskip

      The next calculation shows the cancellation of the leading divergent
      contributions for such solutions: We therefore introduce the new variable
      $\xi = x - s(t)$. The thin--film equation \eqref{tfe-slip-xi} with
      $\eps = 0$ then takes the form
      \begin{align}\label{eq-typebhO2}
  h_t-\dot s h_\xi+(h^3h_{\xi\xi\xi})_\xi \ = \ 0
  \qquad\text{for $\xi>0$}.
      \end{align}
      The next proposition establishes cancellation of the leading divergent
      contributions under explicit assumptions on the profile and its
      derivatives near the contact point.
      \begin{proposition}[Renormalized energy  dissipation]\label{prp-typebhoDISS}
Suppose that $h$ is a sufficiently regular solution of \eqref{eq-typebhO2},
$\dot s>0$ and $\dot s\in C^1_{\rm loc}((0,T))$. Suppose furthermore that
\begin{enumerate}
\item Near the contact point $h$ has the expansion
  \begin{align}\label{eq-typebhO3}
  h(\xi,t) \ = \ (3\dot s)^{\frac13}\xi|\ln\xi|^{\frac13}
  +b \xi|\ln\xi|^{-\frac23}+r(\xi,t),
\end{align}
where $b : (0,T) \to \R$ and as $\xi\to0^+$ for \FA{t \in (0,T)} we have
\begin{align}\label{eq-typebhO13}
  r \ = \ \OO(\xi|\ln\xi|^{-\frac53}), \qquad
  r_\xi \ = \ \OO(|\ln\xi|^{-\frac53}), \qquad
  r_{\xi\xi\xi} \ = \ \OO(\xi^{-2}|\ln\xi|^{-\frac83}).
\end{align}
\item For some $\rho : (0,T) \to \R$ with $\rho > 0$ we have
  \begin{align} \label{eq-typebhO13B} %
    \DS   \int_\rho^\infty h^3h_{\xi\xi\xi}^2\ d\xi \ < \ \infty \qquad\qquad
    \FA{t \in (0,T).}
  \end{align}
\end{enumerate}
Then the renormalized energy dissipation is finite, i.e.
\begin{align}\label{eq-typebhO14}
  \lim_{\d\to0}\left(
    \frac{\dot s}{2}h_\xi(\d,t)^2
    -\int_\d^\infty h^3h_{\xi\xi\xi}^2\ d\xi
  \right) \ \in \ \R \qquad\qquad %
  \FA{t\in(0,T).}
\end{align}
\end{proposition}
\begin{proof}
  Fix $t \in (0,T)$. We write $a := (3\dot s)^{\frac13}$ and $L := |\ln\xi|$. In
  view of \eqref{eq-typebhO13} we obtain
\begin{align}\label{eq-typebhO16A}
  h^3 \ &= \ a^3\xi^3L\left(1+\frac{3b}{a}L^{-1}+\OO(L^{-2})\right), \\
\label{eq-typebhO15b}
  h_\xi \ &= \ aL^{\frac13}+\left(b-\frac a3\right)L^{-\frac23}+\OO(L^{-\frac53}), \\
\label{eq-typebhO16}
  h_{\xi\xi\xi} \ &= \ \frac{a}{3\xi^2}L^{-\frac23}
  \left(1-\frac{2b}{a}L^{-1}+\OO(L^{-2})\right).
\end{align}
Inserting these expansions and using interior regularity we obtain, for fixed
$\rho \in (0,1)$, as $\d\to0^+$,
\begin{align}\label{eq-typebhO17}
  \frac{\dot s}{2}h_\xi(\d,t)^2 \ &= \ \frac{a^5}{6}|\ln\d|^{\frac23}+o(1), \\
\label{eq-typebhO18}
  \int_\d^\rho h^3h_{\xi\xi\xi}^2\ d\xi
  \ &= \ \int_\d^\rho\frac{a^5}{9\xi L^{\frac13}}
  \left(1-\frac{b}{a}L^{-1}+\OO(L^{-2})\right)\ d\xi \ %
  = \ \frac{a^5}{6}|\ln\d|^{\frac23}+C_\rho(t)+o(1).
\end{align}
After subtraction of the leading term, the remaining integrand is
$\OO(\xi^{-1}|\ln\xi|^{-\frac43})$, which is integrable near $\xi=0$.
By assumption \eqref{eq-typebhO13B}, the contribution from $(\rho,\infty)$ is finite and
independent of $\d$. Combining \eqref{eq-typebhO17} and \eqref{eq-typebhO18}
proves \eqref{eq-typebhO14}.
\end{proof}
 Generally, the interfacial energies between liquid and solid might depend on
the horizontal variable. The contact angle then also depends on the position.
In this case, for the rescaled model and for type (b) solutions we would get a
contact angle of the form
$\ang_{\eps}(x) = \psi (x) |\ln \eps|^{\frac 13}\to \infty$ for some function
$\psi = \psi (x)$,$\,$i.e. $\psi (x)$ would modify the angle locally. The energy
has the form
\begin{align}
  \frac12 \int_{s(t)}^\infty h_x^2  \ dx  +  \frac12 |\ln\eps|^{\frac23}  \int_{s(t)}^\infty\psi(x)^2  \ dx.
\end{align}

\section{Solutions in the no slip model} \label{sec-noslip} %

We now discuss some particular problem classes for equation
\eqref{tfe-noslip}. We will consider two types of problems, namely, solutions
with shrinking support and solutions with fixed support.

\subsection{Sufficient conditions for non--spreading of solutions}

We consider sufficient conditions for the non--spreading of solutions. Since we
do not consider the asymptotic behaviour $\eps \to 0$ here, we equivalently
discuss this issue for the equation \eqref{tfe-n} keeping only the dominant term
near the boundary. We will show that the interface is already pinned if the
solution has a finite contact angle or almost linear growth. For $T > 0$ and
$\d>0$, we consider the moving space-time domain
\begin{align}  %
D
\ = \
\{(x,t):s(t)<x<s(t)+\d,\ 0<t<T\}.
\end{align}
We say that $h$ is a local positive solution near the contact point $x=s(t)$ if
the following holds:
\begin{align} \label{HREG}
h
\ \in \
W_{\rm loc}^{3,1}(D),
\qquad
h_t
\ \in \
    L^1_{\rm loc}(D), \qquad
  s\in W_{\rm loc}^{1,1}((0,T)), \qquad %
   h > 0 \text{ in $D$}.
\end{align}
Furthermore,
\begin{align}
  &h_t(\cdot,t)+(h^nh_{xxx})(\cdot,t)_x \ %
    = \ 0
  &&\text{distributionally in $(s(t),s(t)+\d)$}
 \label{eq-nomove1} \\
&h^nh_{xxx}(s(t),t)
\ = \
0
&&\text{for a.e. }t\in(0,T),
 \label{eq-nomove2} \\
&h(s(t),t)
\ = \
0
  &&\text{for all $t \in (0,T)$.}
     \label{eq-nomoveBC}
\end{align}
We assume that the boundary conditions are well--defined as one-sided
  traces. The no--slip paradox is the conjecture that the no--slip condition at
the liquid--solid interface prevents movement of the contact line since this
would lead to infinite dissipation of energy. For the thin--film equation this
can be phrased as follows:
\begin{conjecture}[Finite dissipation implies fixed support] %
  Let $n \geq 3$ and let $h$ satisfy
  \eqref{HREG}--\eqref{eq-nomoveBC}. Furthermore, suppose that %
  \begin{align} \label{S1E8} %
    \iint_{D} h^n  h_{xxx}^{2} \ dxdt \ < \ \infty.  
  \end{align}
  Then $s(t) = s_0$ is constant in $(0,T)$.
\end{conjecture}
The next theorem was discussed during a workshop in the Lorentz Center
\cite{Lorentz-Center}. In particular, J. Eggers asked if under general
smoothness conditions one can prove that the support of the solution does not
expand. We present a short and simple proof with sufficient conditions on the
slope of the profile and its time derivative. Unlike the formal asymptotic
  regimes, the following result is rigorous. 
\begin{theorem}[Almost linear growth implies fixed support] \label{thm-nomove} %
  Let $n \geq 3$ and let $h$ satisfy
  \eqref{HREG}--\eqref{eq-nomoveBC}. Furthermore, suppose that there exist
  functions $C,\gam : (0,T) \to (0,\infty)$ and $\alp\in\mathbb R$, where
  $\alp<\frac13$ if $n=3$, such that, for a.e. $t\in(0,T)$ and for $z\to0^+$ we
  have
\begin{align}
 h(s(t)+z,t)
 \ &\leq \
 C(t)z|\ln z|^\alp,
 \label{eq-nomove3}
 \\
 h_t(s(t)+z,t)
 \ &= \
 -\gam(t)\dot s(t)|\ln z|^\alp
 +o\bigl(|\ln z|^\alp\bigr).
 \label{eq-nomove4}
\end{align}
Then $s(t)=s_0$ is constant
in $(0,T)$.
\end{theorem}
\begin{proof}
  Fix $t\in(0,T)$ such that all the assumptions hold. We write $s=s(t)$ and
  $\gam=\gam(t)$. Arguing indirectly, we assume $\dot s\neq0$. By
  \eqref{eq-nomove1}--\eqref{eq-nomove2} we have
\begin{align}
 (h^nh_{xxx})(s+z,t)
 \ &= \
 -\int_0^z h_t(s+r,t)\ dr \ %
 \upref{eq-nomove4}= \
 \gam\dot s z|\ln z|^\alp\bigl(1+o(1)\bigr).
 \label{eq-nomove6}
\end{align}
Combining \eqref{eq-nomove3} and \eqref{eq-nomove6}, and with the notation
$\sig:=\sign(\dot s)$, we obtain, for some $c = c(t)>0$ and $z_0 = z_0(t)>0$
sufficiently small,
\begin{align}
 \sig h_{xxx}(s+z,t)
 \ \geq \
 c z^{1-n}|\ln z|^{-(n-1)\alp}
 \qquad\qquad %
 \FA{z \in (0,z_0)}.
 \label{eq-nomove7}
\end{align}
Integrating from $z$ to $z_0$ twice, we obtain
\begin{align}
 \sig h_{xx}(s+z,t)
 \ &\leq \
 \sig h_{xx}(s+z_0,t)
 -c\int_z^{z_0} r^{1-n}|\ln r|^{-(n-1)\alp}\ dr \ %
 \leq \
 -c z^{2-n}|\ln z|^{-(n-1)\alp}, \\
 \sig h_x(s+z,t)
 \ &\geq \
 \sig h_x(s+z_0,t)
 +c\int_z^{z_0}r^{2-n}|\ln r|^{-(n-1)\alp}\ dr.
 \label{eq-nomove8}
\end{align}
For $n=3$ we have $\sig h_x(s+z,t) \ \geq \ c|\ln z|^{1-2\alp}$. Integrating
from $0$ to $z$ yields
\begin{align}
 \sig h(s+z,t)
 \ \geq \
 c\int_0^z|\ln r|^{1-2\alp}\ dr
 \ \geq \
 c z|\ln z|^{1-2\alp} \qquad\qquad %
  \text{for $z \ll 1$}.
 \label{eq-nomove13}
\end{align}
If $\sig=-1$, this contradicts $h>0$. If $\sig=1$, it contradicts
\eqref{eq-nomove3}, since $1-2\alp>\alp$ if $\alp < \frac 13$. The case $n>3$ is
analogous: Integrating \eqref{eq-nomove7} three times yields
\begin{align}
  \sig h(s+z,t)
  \ \geq \
  c z^{4-n}|\ln z|^{-(n-1)\alp} \qquad\qquad %
  \text{for $z \ll 1$} 
\end{align}
if $3<n<4$, For $n=4$, the derivative estimate becomes
  $\sigma h_x(s+z,t) \ \geq \ c z^{-1}|\ln z|^{-3\alpha}$.  If
  $\alpha\leq \frac 13$, this lower bound is nonintegrable at zero and
  contradicts the finite zero trace of $h$. If $\alpha>\frac 13$, integration
  gives $\sigma h(s+z,t) \ \geq \ c|\ln z|^{1-3\alpha}$, which contradicts
  positivity when $\sigma=-1$, and the assumed upper bound when $\sigma=1$. For $n>4$, \eqref{eq-nomove8} gives
$\sig h_x(s+z,t) \ \geq \ c z^{3-n}|\ln z|^{-(n-1)\alp}$.
This lower bound is nonintegrable at zero, again contradicting
the finite zero trace of $h$. Hence $\dot s(t)=0$ for a.e. $t\in(0,T)$.
Since $s\in W_{\rm loc}^{1,1}((0,T))$, it follows that $s$ is constant.
\end{proof} 
We note that the above proof does not directly use finite dissipation of
energy.

\subsection{Type (a) solutions -- fixed  support}

In this section, we justify why solutions of \eqref{tfe-noslip} with fixed
support $\supp h = [0,\infty)$ can be expected to exist. We consider the
problem%
\begin{align}  h_t +(h^3h_{xxx})_x\ %
  &= \ 0\ \ \qquad\qquad\text{for }x\in (0,\infty ), \  t>0,  \label{S4E5} \\
  h(0,t) \ %
  &= \ 0\ \ \qquad\qquad\text{for }t\geq 0  \label{S4E6}
\end{align}
with initial data $h_0$. The following linearization suggests the existence
  of solutions to \eqref{S4E5}--\eqref{S4E6} with linear behaviour near the
contact point. More precisely, assume 
\begin{align}   \label{S4E9} %
  h(x,t) \ %
  = \ \gam x + \phi (x,t),
\end{align}
where $\gam$ is a smooth positive function which represents the leading order
behaviour near the contact point. A rigorous construction for solutions of
\eqref{S4E5}--\eqref{S4E6} is work in progress \cite{GKV-preprint}. Here, we
just give formal arguments to show the feasibility of such solutions.

\medskip

Plugging \eqref{S4E9} into \eqref{S4E5} and keeping the leading order terms we
obtain the approximate equation %
\begin{align}
  \phi_t + \gam^3(x^3\phi_{xxx})_x + \dot{\gam} x \ %
  = \ 0.
\end{align}
Using that $\gam >0$ we can change the time scale as
$\tau = \int_{0}^{t}(\gam(s) )^3ds$. Then%
\begin{align}
  \phi_{\tau }+(x^3\phi_{xxx})_x + \frac{\dot{\gam} x}{\gam^3} \ = \ 0 \qquad\qquad %
  \text{for $x>0$,} \label{S5E1}
\end{align}%
where $\dot \gam = \ddt \gam$. To identify the homogeneous boundary modes, we first consider%
\begin{align}
  \phi_{\tau} + (x^3\phi_{xxx})_x \ = \ 0 \ \qquad %
  \text{for $x>0$.}
\label{S5E3}
\end{align}
Equation \eqref{S5E3} is a fourth order parabolic equation that degenerates at
$x=0$. Let $\tilde{\phi}$ denote the Laplace transform of $\phi$ in time with
Laplace variable $z \in \C$. Then \eqref{S5E3} turns into the ODE 
\begin{align} \label{this-ODE} 
  z \tilde{\phi}+(x^3\tilde{\phi}_{xxx})_x \ 
  = \ \phi_0\qquad\qquad 
  \text{for $x>0$ and $\Re z>0$ sufficiently large}
\end{align}
where $\phi_0$ is the initial datum for the solution of \eqref{S5E3}. The
  homogeneous equation associated with \eqref{this-ODE} has the leading boundary
  modes $\ln x$, $1$, $x$ and $x^2$. The modes $\ln x$ and $1$ are excluded by
the boundary condition $\phi(0,\tau)=0$. The admissible homogeneous modes
therefore have leading behaviour $x$ and $x^2$. This mode calculation
  concerns the homogeneous equation; the forcing must be restored when
  determining the expansion for $h$. For the inhomogeneous equation
\eqref{S5E1}, a linear contribution to $\phi$ can formally be absorbed into
  $\gam$ by adjusting the decomposition \eqref{S4E9}.  The remaining correction
must balance the forcing $\dot\gam x/\gam^3$. A quadratic term alone cannot
provide this balance, since its third derivative vanishes. Instead, we make the
ansatz $\phi \approx B x^2 \ln x + C x^2$ and use $(x^3(x^2\ln
x)_{xxx})_x=4x$. Matching the terms of order $x$ in \eqref{S5E1} then yields the
expansion
\begin{align}\label{eq-fixed-support-correction}
  \phi(x,t) \ = \
  -\frac{\dot\gam(t)}{4\gam(t)^3}x^2\ln x+\OO(x^2).
\end{align}
Here, the time derivative of the displayed correction is of order
$x^2|\ln x|=o(x)$ for smooth $\gam$, and hence does not contribute to
this leading balance. In particular, $\phi=o(x)$ and
$h\approx\gam x$ as $x\to0^+$. This calculation establishes formal
local consistency, not existence.

\subsection{Type (b) solutions -- prescribed shrinking support \label{IncreasingS}}

We now discuss a different class of solutions of \eqref{tfe-noslip}, supported
in the half-line $(s(t) ,\infty)$ for $t \geq 0$. The leading local asymptotics
are formally compatible with a sufficiently smooth and strictly increasing
contact-point $s(t)$.  We assume a single contact point, i.e.
$\supp h = [s(t),\infty)$. With $\xi = x - s(t)$, \eqref{tfe-noslip} takes
the form
  \begin{align} \label{tfe-noslip-xi} 
    h_t  - \dot s h_\xi +(h^3 h_{\xi\xi\xi})_\xi \ &= \ 0  \qquad 
    \text{for $\xi > 0$}, \\
    h(0,t) \ &= \ 0
  \end{align}
  with initial data $h(\cdot,0) = h_0$. We write by some abuse of notation $h$
  as a function of $\xi$ and $t$. We look for solutions of \eqref{tfe-noslip-xi}
  with the asymptotic behaviour
\begin{align}   \label{ass-type-b} 
  h(\xi,t) \ \approx \ \gam \xi  (- \ln \xi)^{\frac 13} \qquad\qquad 
  \text{as } \xi \ll 1.
\end{align}
Plugging the asymptotics \eqref{ass-type-b} into \eqref{tfe-noslip-xi} we obtain
the relation $3 \dot s \ = \ \gam^3$ i.e. we get the asymptotics
\eqref{def-typeb} of the type (b) solutions. The non--negativity of the solution
together with \eqref{ass-type-b} implies the consistency condition $\dot s
>0$. Solutions with the asymptotics \eqref{ass-type-b} have been investigated
for travelling solutions e.g. in \cite{BoattoEtal-1993,BowenKing-2001}. We
remark that these solutions do not have a finite contact angle.

\medskip

Linearizing formally around the asymptotic behaviour \eqref{ass-type-b} we
obtain a linear equation near the contact point which is solvable in the class
of functions with linear behaviour including a logarithmic correction. We write
\begin{align}   \label{H-bar-def} 
  h(\xi,t) \ 
  =: \ \OL{h} ( \xi,t) +u( \xi,t), \qquad\qquad 
  \text{where} \quad 
  \OL{h}(\xi,t) \ 
  := \ \gamma \xi ( - \ln  \xi) ^{\frac 13}
\end{align}
assuming that $u(\xi,\cdot) \ll \OL{h}(\xi,\cdot)$ for $\xi \ll 1$ with
similar estimates for spatial derivatives. In leading order we obtain the
approximated equation as $\xi\to 0^+$
\begin{align}
  \OL{h}_t +u_t + \big[ ( \OL{h}^{3}\OL{h}_{\xi\xi\xi})_\xi -\dot s \OL{h}_{\xi}\big] +( \OL{h}^{3}u_{\xi\xi\xi}+3\OL{h}^{2}
  \OL{h}_{\xi\xi\xi}u-\dot s   u)_\xi \ 
  = \ 0.  \label{u-nonhomog}
\end{align}
The leading order behavior of $u$ as $\xi\to 0^+$ can be computed arguing
similarly as in the case of fixed, prescribed support (Subsection 3.2). We need
to examine the contribution of the homogeneous terms, namely
\begin{align}
u_t + \big( \OL{h}^{3}u_{\xi\xi\xi}+3\OL{h}^{2}\OL{h}_{\xi\xi\xi}u-\dot s  u\big)_\xi \ 
= \ 0.  \label{u-homog}
\end{align}
A major difference with the case of constant support is the fact that the term
$\dot s u$ yields a contribution comparable to the term
$\OL{h}^{3}u_{\xi\xi\xi}$. The leading order terms for $\xi\ll 1$ in
\eqref{u-homog} are
\begin{align}
  u_t + \gamma^3 \big( \xi^{3} (- \ln \xi) u_{\xi\xi\xi}+ \frac{2}{3}u\big)_\xi \ 
  = \ 0 \qquad 
  \text{for $\xi \ll 1$}.
\label{u-homog-app}
\end{align}
The four possible asymptotic behaviors for $\xi \ll 1$ for solutions of
\eqref{u-homog-app} are $a_1 (- \ln \xi) ^{\frac 13}$, $a_2$,
$a_3 \xi (- \ln \xi)^{-\frac{2}{3}}$ and $a_4 \xi^{2}(- \ln \xi)^{\frac 13}$ for
suitable functions $a_i = a_i(t)$. The above mode analysis suggests that
solutions can be determined by prescribing two boundary conditions at
$\xi=0$. With our assumptions we get $a_1 = a_2 = 0$. Therefore, the first
contribution due to the homogeneous part of \eqref{u-nonhomog} has the behaviour
$a_3 \xi (- \ln \xi)^{-\frac{2}{3}}$ which satisfies the consistency condition
$u(\xi,\cdot) \ll \OL{h}(\xi,\cdot)$ for $\xi \ll 1$.

\medskip

We remark that the solvability of \eqref{tfe-noslip} in a domain
$\{ (x,t) :t>0,\ x>s(t)\} $ with the behaviour \eqref{ass-type-b} can be thought
of as a boundary value problem, with unusual boundary conditions that are due to
the degeneracy of \eqref{tfe-noslip} near the contact line $x=s(t)$. Thus
  the local expansion does not determine $s(t)$ and is formally compatible with
  a prescribed shrinking support. These solutions are the solutions that have
been termed as type (b) solutions in the introduction. Solvability of this
problem will be addressed in future work.

\section{Asymptotics for weak slippage} \label{sec-slip} %

We consider the equation \eqref{tfe-slip} with weak slippage. In the first
subsection we discuss well--posedness for the equation. We then give formal
  matching arguments for regimes (a) and (b). The conjectural regime (c) is not
  analyzed further in this paper. For simplicity we assume
$\supp h = [s(t),\infty)$ and we use the variables $\xi = x - s(t)$. Then
\eqref{tfe-slip} takes the form
  \begin{align} \label{tfe-slip-xi} %
    h_t  - \dot s h_\xi +(( h^3 + \eps^{3-n} h^n) h_{\xi\xi\xi})_\xi \ &= \ 0  \qquad 
    \text{for $\xi > 0$}, \\
    h(0,t) \ = \ 0, \quad h_\xi(0,t)  \ &= \ \ang_\eps
  \end{align}
  with initial data $h(\cdot,0) = h_0$. By an abuse of notation we write 
  $h$ as a function of $\xi$ and $t$.

\subsection{Well--posedness and Boundary conditions}

We consider well--posedness and boundary conditions for the problem
\eqref{tfe-slip-xi}. Since we do not consider the asymptotic behaviour
$\eps \to 0$ here, we can equivalently discuss the issue for the equation
\eqref{tfe-n}.

\medskip

\textbf{Solutions with prescribed contact line $s(t)$.} One can construct
some solutions with prescribed position of the contact points. For simplicity we
consider the case when $\spt h = [s(t),\infty)$. The difference is
  that for $n<3$ the local asymptotic expansion is formally compatible
  with a prescribed smooth function $s(t)$, without a sign restriction on
  $\dot s$. The behaviour of the solution $h(x,t)$ near the contact point
$x=s(t)$ with $\xi := x - s(t)$ is given by
\begin{align}
  h(x,t) \ = \ \gam \xi  + \phi (\xi,t), \qquad %
\end{align}%
for some smooth function $\gam >0$ and some function $\phi$ with
\begin{align} \label{ass-phi} %
  \phi(\xi,t) = o(\xi ), \ \  %
  \phi_\xi(\xi,t) = o(1), \ \  %
  \phi_t(\xi,t) = o(\xi) \qquad %
  \text{as $\xi\to 0^+$}. 
\end{align}
We examine the behaviour of the linearized problem satisfied by $\phi $ near
the contact point. Keeping the leading order terms we obtain%
\begin{align}
  -\gam \dot s  + \phi_t -\dot s  \phi_\xi + \gam^n (\xi^n \phi_{\xi\xi\xi})_\xi  \ %
  =  \ 0 \qquad\qquad %
  \text{for $\xi >0, t>0$},
\end{align}%
with $\phi (0,t) =0$. Since the linear part $\gam\xi$ has vanishing third
derivative, the no-flux condition \eqref{no-massflux} gives
$\xi^n\phi_{\xi\xi\xi}\to0$ as $\xi\to0^+$. Integrating the equation in $\xi$
and using \eqref{ass-phi} we obtain the following asymptotic behaviour as $\xi \to 0$:
\begin{align}   \label{S5E6} %
  \phi(\xi,t) \ = \
  \left \{
  \begin{array}{ll}
      \DS C_n \gam^{1-n}\dot s  \xi^{4-n} + o(\xi^{4-n}) \ \ \ \qquad %
    &\text{if $2 < n < 3$,} \vspace{0.7ex}\\
    \DS \tfrac 12 \gam^{-1}\dot s  \xi^2 \ln \xi + o(\xi^2 \ln \xi) \ \ \ \qquad %
    &\text{if $n = 2$,} \vspace{0.7ex} \\
    \DS \tfrac 12 \phi_{\xi\xi}(0) \xi^2 + o(\xi^2).
    &\text{if $0 < n < 2$}.
  \end{array}
  \right.
\end{align}%
with $C_n := [(4-n)(3-n)(2-n)]^{-1}$.  Notice that there are no constraints about the
sign of $\dot s$. This differs from the type (b) no-slip expansion, which is
  formally compatible only with shrinking support (cf.  Section
\ref{IncreasingS}).

\medskip

\textbf{Solutions with prescribed contact-line angle.} Alternatively, one
  may consider a boundary problem for \eqref{tfe-n} with $n<3$ in which the
  contact angle $\gam\geq0$ is prescribed and $s(t)$ is an unknown. Related
  existence results can found, e.g., in
\cite{BertozziPugh-1994,BerettaBertschDalpasso-1995,BertschDalpassoGarckeGruen-1998,DalpassoGarckeGruen-1998,Gruen-2004-1,GiacomelliKnuepferOtto-2008,GnannWisse-2022}. It
seems possible to use an approach similar to the one in those papers to solve
the problem with prescribed, smooth $\gam$ by reformulating the angle condition
as an integral equation for $s(t)$.
\begin{remark}[Boundary conditions] %
  As discussed before, for our boundary value problem, we can either prescribe
  two boundary conditions for $h$ or we can prescribe the contact point velocity
  and one boundary condition. We show how a similar issue arises from the
  classical one--dimensional Stefan problem
    \begin{align} \label{stefan}
      &u_t  \ = u_{xx}\ \ 
      &&x>s(t),  \ t \geq 0, \\
      &u  \ = \  1, \quad  \dot s \ = \ - u_x, 
      &&x = s(t),  t \geq 0
    \end{align}
    with initial conditions $u = u_0$. In terms of $U = \int_x^\infty u$ for
    $x > s(t)$ we obtain
    \begin{align} \label{stefan-2} 
      &U_t  \ = U_{xx}\ \ 
      &&x>s(t),  \ t \geq 0, \\
      &U  \ = \  \int_{s_0}^\infty u_0 \ dx, \quad  U_x \ = \ - 1,  
      &&x = s(t),  t \geq 0. \label{U-bc}
    \end{align}
    Indeed, for the first boundary condition in \eqref{U-bc} we calculate
    \begin{align}
      \ddt U(s(t), t) \ 
      &= \ U_x(s(t),t) \dot s  + \p_t U (s(t),t) \ 
        = \ - \dot s  + \int_{s(t)}^\infty u_{xx} (\xi,t) \ d\xi \\ 
      &= \ - \dot s  - u_{x} (s(t),t) \ = \ 0. 
    \end{align}
    This illustrates that the Stefan problem \eqref{stefan} -- one boundary
    condition for $u$ and the speed $\dot s$ is prescribed --- is equivalent to
    the boundary value problem \eqref{stefan-2}--\eqref{U-bc} where two
    boundary conditions on the function are assumed at the free boundary while
    the speed of the contact line is only implicitly given.
\end{remark}

\subsection{Formal asymptotics of type (a) solutions -- Cox-Voinov
  law} \label{ss-TannerCV}

In this subsection we discuss the type (a) solutions and their relation to the
Cox-Voinov law which connects the speed of the contact point with the
  apparent contact angle.  We will see that for $\eps \to 0$, type (a)
solutions of \eqref{tfe-noslip} arise as limit of solutions of
\eqref{tfe-slip-xi} where the constant angle $\ang_\eps$ does not depend on
$\eps$. Both in the case of complete or partial wetting, the Cox--Voinov law can
be reformulated as a law for the speed of the contact point. Our main goal is to
illustrate the fact that those computations yield a limit law for the speed of
the contact lines associated to the equation \eqref{tfe-noslip} obtained as a
limit of \eqref{tfe-slip}. In this regime the contact line moves very slowly
with speed of order $|\ln \eps|^{-1}$.

\medskip

We consider solutions of \eqref{tfe-slip-xi} with a single contact point.  Far
from the contact point, the solutions approximately solve \eqref{tfe-noslip} and satisfy
\begin{align} \label{S5E8} %
  h(\xi,t) \ %
  \lupref{S1E4}\approx \ \gam \xi, \qquad\qquad \text{for $\eps \ll \xi \ll 1$},
\end{align}%
where $\gam >0$ is an $\OO(1)$ function that depends on
initial data. The asymptotic formula \eqref{S5E8} breaks down near the contact line when
$h^3 \sim \eps^{3-n}h^{n}$, i.e.  $\xi \sim \eps$. To investigate the solutions in the inner
region we use the new variables%
\begin{align} \label{inner-var-a} %
  h \ = \ \eps H,\ \quad \ \xi \ = \ \eps y,\ \quad  t \ = \ \eps \tau  
\end{align}
In terms of these new variables \eqref{tfe-slip-xi} takes the form
\begin{align}
  H_{\tau }-\dot s H_y+ \left(\left(H^3+H^{n}\right)
  H_{yyy}\right)_y \ %
  = \ 0 \qquad\qquad
  \text{for $y>0$}. 
\end{align}
with $H_y(0,\cdot) = \ang$.  We expect the solutions to become stationary if
$\tau $ is large. With the assumption \eqref{no-massflux} we obtain the
approximative equation
\begin{align}   \label{S5E9} %
  -\dot s + (H^2+H^{n-1}) H_{yyy} \ = \ 0 \qquad\qquad
  \text{for $y>0$}.
\end{align}
We need to match solutions of \eqref{S5E9} with the asymptotic behaviour
\eqref{S5E8} for $\eps \ll \xi\ll 1$ where $\dot s$ is one of the quantities
that must be determined. The equation \eqref{S5E9} must be solved with contact
angle $\ang = 0$ in the complete wetting case and contact angle $\ang > 0$ in
the partial wetting case. This analysis leads to the Cox--Voinov law in the case of
complete wetting as well partial wetting. We use the self--consistent assumption
$|\dot s| \ll 1$ for $\eps \ll 1$.

\medskip

\textbf{Partial wetting case.} We now consider solutions of \eqref{S5E9} with
nonzero contact angle $\ang > 0$ assuming that the slope is almost constant for
distances to the contact point of order $\eps$. We assume $n \in (1,3)$ to be
fixed and consider the asymptotics $\eps \to 0$.  We make the ansatz
  $H=H_0+H_1+\ldots$ and assume $|H_1|\ll H_0$ in the matching region. We
obtain to leading order%
\begin{align}
  H  \ %
  \approx  \ \ang y \qquad\qquad \text{as $y\to 0^+$}.
\end{align}
Furthermore, $H_1$ solves to leading order
\begin{align} \label{H1-form-before}
  H_{1,yyy} \ %
  = \ \frac{\dot s}{H_0^2+H_0^{n-1}} \ %
  = \ \frac{\dot s}{\ang^2 y^2  + \ang^{n-1} y^{n-1}} \qquad\qquad %
  \text{for $y>0$.} 
\end{align}
For an advancing left contact point in the present orientation, we have
  $\dot s\leq0$. Choosing the integration constants so that $H_1=o(y)$ as
$y \to 0^+$ and $H_{1,yy}(y) \to 0$ as $y\to\infty$ we obtain
\begin{align} \label{H1-form} %
  H_1 \ %
  = \ |\dot s| \int_0^y \int_0^{y_1} \int_{y_2}^{\infty }\frac{1}{\ang^2 y_3^2 +
  \ang^{n-1} y_3^{n-1}} \ dy_3 dy_2 dy_1
\end{align}
noting that $\dot s \leq 0$.  Integrating \eqref{H1-form} we get the
asymptotic behaviour
\begin{align}
  H_1 \ %
  \approx \ \frac{|\dot s|}{\ang^2}y\ln y \qquad\qquad %
  \text{as $y\to \infty$.}
\end{align}
Using the approximation $H = H_0 + H_1$ with $|H_1| \ll H_0$ and in terms of the
variables $h$, $\xi$, cf. \eqref{inner-var-a} for $\eps \ll \xi \ll 1$ with
  $|\ln\xi| \ll |\ln\varepsilon|$ we then obtain%
\begin{align} \label{Tan-outer-match} %
  h \ = \ \eps H \ %
  \approx \ \eps\Big( \ang y+ \frac{|\dot s|}{\ang^2}y\ln y \Big) \ %
  \approx \ \ang\xi + \frac{|\dot s|}{\ang^2} |\ln \eps| \xi  \qquad\qquad %
\end{align}
Matching \eqref{Tan-outer-match} with \eqref{S5E8} we obtain the linearized
Cox--Voinov law \cite{Voinov-1977,Cox-1986,Hocking-1992} %
\begin{align} \label{S7E4} %
  \dot s \ %
  = \ \frac{\ang^2}{|\ln \eps|} (\ang-\gam),
\end{align}
which is formally valid for $n\in(1,3)$ in the perturbative regime
  $|\gam-\ang|\ll\ang$. Notice that the speed of the contact line vanishes if
the outer contact angle $\gam$ is the same as the microscopic contact angle
$\ang$. The contact line speed logarithmically converges to zero if
$\eps \to 0$. Hence, in real situations the change of the droplet profile $h$
and the motion of the contact line $s(t)$ could take place in similar time
scales. We note that the above computation has some analogies with the
computation in Subsection \ref{MatchTypeb}.

\medskip

\textbf{Complete wetting case.}  Although most of the paper is concerned the
  non--zero contact angle case, for the readers' convenience we also give a
  short motivation of Tanner's Law which is relevant for the zero contact angle
  case. We consider $n \in [2,3)$. We need to obtain a solution of \eqref{S5E9}
  with a suitable choice of $\dot s$ and the boundary condition
  $H_y(0,\tau) \ = \ 0$ such that $h \ = \ \eps H$ satisfies the matching
  condition \eqref{S5E8}. Near the contact point we have $H^{n-1} \gg
  H^2$. Hence the leading-order contribution $H_0$ is determined by
  \begin{align} %
    H_0^{n-1}H_{0,yyy} \ = \ \dot s.
  \end{align}
Looking for a power-law solution gives
\begin{align}\label{S6E3}
  H(y,\tau) \ %
  \approx \ C_n(-\dot s)^{\frac1n}y^{\frac3n},
  \qquad
  C_n \ = \
  \Big(\frac{n^3}{3(3-n)(2n-3)} \Big)^{\frac1n} > 0
  \qquad
  \text{as $y\to0^+$}.
\end{align}
In particular, for $n \in [2,3)$ we get $\dot s<0$ so that the support of the
solution can only expand. We rescale variables to eliminate the dependence on
the contact-line speed. We set
\begin{align}   \label{S6E6} %
  \eta \ := \ (-\dot s)^{\frac13}y.
\end{align}
Then \eqref{S6E3} takes the form $H(\eta) \approx C_n\eta^{\frac3n}$ as
$\eta\to0^+$. In terms of $\eta$, equation \eqref{S5E9} becomes
\begin{align}\label{S6E5} %
  1+(H^2+H^{n-1})H_{\eta\eta\eta} \ = \ 0 \qquad\qquad \text{for $\eta>0$}.
\end{align}
For $\eta \gg 1$, we expect that the solution enters a regime where $H \gg
1$. Since $n<3$, in this regime we have $H^{n-1}\ll H^2$ and hence \eqref{S6E5}
reduces to in leading order to
\begin{align}
  1+H^2H_{\eta\eta\eta} \ = \ 0.
\end{align}
With the ansatz $H \approx A \eta (\ln \eta)^{\frac 13}$ we obtain
$A = 3^{\frac 13}$ and hence obtain
\begin{align}\label{S6E7}
  H \ %
  \approx \ 3^{\frac13}\eta(\ln\eta)^{\frac13} \qquad\qquad \text{for $\eta \gg 1$}.
\end{align}
Returning to the original variables using \eqref{inner-var-a} and \eqref{S6E6},
the formula \eqref{S6E7} becomes
\begin{align}\label{S7E1} %
  h \ %
  \approx \ 3^{\frac13}(-\dot s)^{\frac13}\xi \Big( \ln\frac{(-\dot s)^{\frac13}\xi}{\eps} \Big)^{\frac13} \ %
  \approx \ 3^{\frac13}(-\dot s)^{\frac13}\xi |\ln \eps|^{\frac13} \quad \text{for $\eps(-\dot s)^{-\frac13}\ll\xi\ll1$},
\end{align}
where we used the assumption $|\ln (-\dot s)^{\frac 13}\xi| \ll \ln \frac 1\eps$
for the last approximation which is consistent with \eqref{S7E2}. Matching the behaviour \eqref{S7E1} with \eqref{S5E8} gives
  \begin{align}\label{S7E2}
    \dot s \ %
    = \ -\frac{\gam^3}{3\ln\frac1\eps} \qquad\qquad \text{as $\eps\to0$}.
  \end{align}
  The leading-order Tanner--Cox--Voinov law is hence independent of the slip
  exponent $n\in[2,3)$ since it is determined by the no-slip regime $H\gg1$.

\subsection{Formal asymptotics of type (b) solutions\label{MatchTypeb}}
  We next consider the limit $\eps \to 0$ for solutions of \eqref{tfe-slip} in
  the case of type (b) solutions. We consider \eqref{tfe-slip-xi} near a single
  contact point and assume $n \in (0,3)$. In the outer region, where
  $h^3\gg\eps^{3-n}h^n$, we assume that the solution approximately has the type
  (b) behaviour \eqref{ass-type-b}. Correspondingly, we consider the scaling
\begin{align}\label{eq-typeb1}
  \ang_\eps \ \approx \ \gam |\ln \eps|^{\frac13}
\end{align}
for the time-independent contact angle with fixed $\gam>0$.
  
\medskip

The transition between inner and outer region occurs where the two mobilities
are of the same order, i.e. where $h^3 \sim \eps^{3-n}h^{n}$. In view of the
type (b) expansion \eqref{ass-type-b}, the size $\Delta_\eps$ of the inner
region is given by $\Delta_\eps^{3-n} (- \ln \Delta_\eps) \ %
\sim \ \eps^{3-n} (-\ln \Delta_\eps)^{\frac{n}{3}}$.  Since
  $|\ln \Delta_\eps| \sim |\ln \eps|$ as $\eps \to 0$, we obtain
\begin{align}
  \Delta_\eps
  \ \sim \
  \eps |\ln \eps|^{-\frac13}.
\end{align}
The inner scaling describes distances $\xi\sim\Delta_\eps$, whereas the
  outer type (b) expansion applies for $\Delta_\eps\ll\xi\ll1$.  To match the
  two expansions, we consider the large-$y$ behaviour of the inner solution in
  an overlap region (identified below) where both asymptotic descriptions are
  expected to apply.  We therefore introduce for some fixed time $t_0 > 0$ the
inner variables
  \begin{align}     \label{inner-var-b} %
    \xi \ =: \ \eps |\ln \eps|^{-\frac 13}y, \quad %
    t \ =: \ t_{0}+ \eps |\ln \eps|^{-\frac 43}\tau,\quad %
    h \ =: \ \eps H,\  \ %
  \end{align}
  In terms of the inner variables \eqref{inner-var-b}, equation
  \eqref{tfe-slip-xi} takes the form
  \begin{align}
    &H_{\tau }-|\ln \eps|^{-1}  \dot sH_{y}+((H^3+H^{n})H_{yyy})_{y} \ = \ 0, \qquad %
    &&\text{for $y>0$},  \label{T4E2} \\
    &H \ = \ 0,\ \ H_{y} \ = \ \ang_\eps|\ln\eps|^{-\frac13} \ \ %
    &&\text{for $y=0$,}  \label{T4E3-2}
  \end{align}
  where, as always, $\dot s=\frac d{dt}s$. Assuming formally that $H_\tau$
    is negligible at order $|\ln\eps|^{-1}$ on the matching scales we then
  obtain
\begin{align}
  &-|\ln \eps|^{-1}\dot sH_{y}+((H^3+H^{n})H_{yyy})_{y} =0, \qquad %
    &&\text{for $y>0$.}  \label{T4E4}
\end{align}
Integrating \eqref{T4E4}, by the no-mass-flux condition at $y=0$, and
  dividing by $H$, we get
\begin{align}
  -|\ln \eps|^{-1}\dot s+(H^2+H^{n-1})H_{yyy} \ \lupref{no-massflux}= \ 0,%
  &&\text{for $y>0$}
  \label{T5E3}
\end{align}
Since $H_y(0,\tau) \to\gam$ as $\eps\to0$ by \eqref{eq-typeb1}, we formally
obtain the leading-order problem
\begin{align}
  &(H^2+H^{n-1})H_{yyy} \ = \  0,\ \ &&\text{for $y > 0$},\ \\
  &H=0,\ \ H_{y} = \gam \ \ \ &&\text{for $y=0$.} \label{eq-H2} %
\end{align}
Matching with the outer type (b) profile excludes the quadratic mode of the
leading-order inner equation, yielding the linear profile $H(y)=\gam y$. To
determine the first correction to the linear inner profile, we write
\begin{align}
  H  \ %
  =: \ \ang_\eps|\ln\eps|^{-\frac13}y+ \frac 1{|\ln \eps|}U(y)  \qquad\qquad %
  \text{for $y > 0$.} \label{T5E2}
\end{align}%
satisfying $U(0)=U_y(0)=0$. Substituting \eqref{T5E2} into \eqref{T5E3}, using
$\ang_\eps|\ln\eps|^{-\frac13}\to\gam$, and collecting the terms of order
$|\ln\eps|^{-1}$, we obtain
\begin{align}
  U_{yyy} \ = \ \frac {\dot s}{\gam^2 y^2+ \gam^{n-1} y^{n-1}} \qquad\qquad %
  \text{for $y > 0$}.
\end{align}%
By assumption, we have $U_{yy} \to 0$ for $y \to \infty$. Integrating we hence
get
\begin{align} \label{eq-Uyy} %
  U_{yy}(y) \ %
 = \ - \frac {\dot s}{\gam^2} Q_\gam(y), \qquad
  \text{where } Q_\gam(y) :=  \int_{y}^{\infty }\frac 1{z^2+ \gam^{n-3}z^{n-1}} \ dz, \ %
\end{align}
noting that the integral is finite if $n < 3$. Furthermore, 
\begin{align} \label{lob-3} %
  \Big|Q_{\gam }(y) - \frac 1{y} \Big| \ %
  = \ \int_{y}^{\infty }\frac {\gam^{n-3}z^{n-1}}{(z^2+ \gam^{n-3}z^{n-1})z^2} \
  dz \ %
  = \ \OO(y^{n-4}) \qquad 
  \text{for  $y \gg 1$.}
\end{align}
Integrating \eqref{eq-Uyy} twice using $U(0) = U_y(0) = 0$ yields%
\begin{align}\label{lob-1}
  U(y)
  \ &= \
  -\frac{\dot s}{\gam^2}
  \int_0^y
  \int_0^{z_1}
  Q_\gam(z_2)
  \ dz_2
  \ dz_1 \ %
  = \
  -\frac{\dot s}{\gam^2}y\ln y
  + \OO(y)
  \qquad
  \text{as $y\to\infty$}.
\end{align}
The correction is smaller than the leading inner profile
  $\gam y$ provided $\ln y=o(|\ln\eps|)$. To match the inner solution with the
  outer type (b) expansion we expand the outer profile. Leading-order matching
  identifies the coefficient $\gam$ in \eqref{ass-type-b} with the limit
  $\ang_\eps|\ln\eps|^{-\frac13}\to\gam$ in \eqref{eq-typeb1}. We have 
\begin{align} \label{T4E6} %
  H \ %
  \lupref{inner-var-b}= \ \frac h\eps \ %
  &\lupref{ass-type-b}\approx \ \frac{\gam \xi}\eps |\ln \xi|^{\frac 13} \ %
  \lupref{inner-var-b}= \
  \gam y
  \left(
    1+\frac{\frac13\ln|\ln\eps|-\ln y}{|\ln\eps|}
  \right)^{\frac13}
  \\
  &\ = \
  \gam y
  +\frac{\gam y\ln|\ln\eps|}{9|\ln\eps|}
  -\frac{\gam y\ln y}{3|\ln\eps|}
  +\OO\left(
    \frac{y(\ln|\ln\eps|+\ln y)^2}{|\ln\eps|^2}
  \right).
\end{align}%
Here we used $|\ln\xi|=|\ln\eps|+\frac13\ln|\ln\eps|-\ln y$. For $1\ll y$
  and $\ln y=o(|\ln\eps|)$ we hence have the linear behaviour
  $H(y) \ \approx \gam y$. To determine the velocity, we formally match the
  second derivatives of the inner and outer expansions, thereby eliminating the
  terms linear in $y$. Assuming that the expansions can be differentiated twice
  in the overlap region, we obtain
  \begin{align}\label{eq-typeb-curvature-match}
  -\frac{\gam}{3|\ln\eps|y} \ %
  \lupref{T4E6}\approx \ H_{yy} \ %
  \approx \  \frac{U_{yy}}{|\ln \eps|} \ %
  \lupref{lob-1}\approx \ -\frac{\dot s}{\gam^2|\ln\eps|y} \qquad\qquad %
\end{align}
Matching the coefficients gives the velocity formula \eqref{dots-formula}.

\subsection{Higher-order asymptotics for type (b) solutions}
\label{sec-typeb-higher-order}

We compute one additional term in the asymptotic expansion of $h$ of the
  type (b) profile. This is motivated by the homogeneous mode analysis in
  Subsection~3.3 and is consistent with the expansion used in
  Proposition~\ref{prp-typebhoDISS}. We use the coordinates $\xi := x-s(t)$ with
  $\eps = 0$ and hence consider \eqref{eq-typebhO2}. We assume $\dot s>0$ and
  $\dot s\in C^1_{\rm loc}((0,T))$, and write $a(t):=(3\dot s(t))^{\frac13}$ and
  $L:=|\ln\xi| = - \ln \xi$. We make the ansatz
\begin{align} \label{eq-typebhOansatz} %
  h(\xi,t) \ %
  = \ a(t)\xi L^{\frac13}+b(t)\xi L^{-\frac23}
  +K(t)\xi L^{-\frac53}+o(\xi L^{-\frac53}) \ %
  =: \ \bar h + \varphi
\end{align}
with $\bar h:=a\xi L^{\frac13}$. We assume formally that the expansion can be
differentiated term by term in $\xi$, that the coefficients are sufficiently
regular and that $h_t=o(L^{-\frac53})$ as $\xi\to0^+$.  Substituting
\eqref{eq-typebhOansatz} into \eqref{eq-typebhO2} and differentiating formally
term by term, we obtain
\begin{align}\label{eq-typebhO6}
  \begin{aligned}
    (- \dot s\varphi +\bar h^3\varphi_{\xi\xi\xi}
    +3 \bar h^2\bar h_{\xi\xi\xi}\varphi)_\xi %
    +3(\bar h^2\varphi\varphi_{\xi\xi\xi}
  +\bar h\bar h_{\xi\xi\xi}\varphi^2)_\xi %
  = \ (\dot s\bar h - \bar h^3\bar h_{\xi\xi\xi})_\xi  +o(L^{-\frac53}).
  \end{aligned}
\end{align}
Using $L_\xi=-\frac 1{\xi}$, $\bar h:=a\xi L^{\frac13}$, 
$\bar h_{\xi\xi\xi}=\frac{a}{3\xi^2}L^{-\frac23}
-\frac{10a}{27\xi^2}L^{-\frac83}$ and $a^3=3\dot s$ we get
\begin{align}\label{eq-typebhO9}
  (\dot s\bar h - \bar h^3\bar h_{\xi\xi\xi})_\xi  \ %
  \ &= \ \frac{10}{27}a^4 (\xi L^{-\frac53})_\xi
  \ = \ \frac{10}{27}a^4L^{-\frac53}+\frac{50}{81}a^4L^{-\frac83}.
\end{align}
On the other hand, the linear terms on the left hand side of \eqref{eq-typebhO6}
contribute $-a^3K L^{-\frac53}$: the contributions of $b\xi L^{-\frac23}$
cancel at leading order and their remaining linear contribution is
$O(L^{-\frac83})$.  The two quadratic fluxes have leading terms
$-2a^2b^2\xi L^{-\frac53}$ and $a^2b^2\xi L^{-\frac53}$, respectively. Hence
\begin{align}  %
  \begin{aligned}
  &-\dot s\varphi_\xi+(\bar h^3\varphi_{\xi\xi\xi})_\xi
  +3(\bar h^2\bar h_{\xi\xi\xi}\varphi)_\xi
    +3(\bar h^2\varphi\varphi_{\xi\xi\xi}
    +\bar h\bar h_{\xi\xi\xi}\varphi^2)_\xi %
    = -a^3K L^{-\frac53} - a^2b^2L^{-\frac53} +o(L^{-\frac 53}).
  \end{aligned}
\end{align}
Comparison with \eqref{eq-typebhO6} yields $K=-\frac{10}{27}a-\frac{b^2}a$. We
thus obtain the formal expansion
\begin{align}\label{eq-typebhO4}
  h(\xi,t) \ = \ a\xi L^{\frac13}+b\xi L^{-\frac23}
  -\left(\frac{10}{27}a+\frac{b^2}{a}\right)\xi L^{-\frac53}
  +o(\xi L^{-\frac53}).
\end{align}

\section{Conclusion} %

We analyze the evolution of thin liquid droplets in the lubrication
approximation with different slip conditions at the liquid--solid
interface. Motivated by the classical no-slip paradox which states that the
Navier–Stokes equations with a no-slip boundary condition require unphysical
infinite dissipation during droplet spreading, we focus on the limit of
vanishing slip.  Our analysis identifies two distinguished physically
  relevant formal limiting regimes and suggests a third conjectural regime, in
  which the solutions approximate qualitatively different evolutions of the free
  boundary. These regimes correspond to distinct scalings of the microscopic
  contact angle as the regularization parameter vanishes:
\begin{itemize}[leftmargin=5ex]
\item[(a)] The solutions in this class are regular at the contact line and have
  fixed support. This reflects the well known fact that the contact line remains
  fixed for sufficiently regular no-slip solutions.
\item[(b)] For the second solution class, the contact line recedes on a time
  scale of order one while the solution's profile exhibits a characteristic
  logarithmic correction near the moving contact point. Intriguingly, for these
  solutions the dissipation near the contact point diverges, but the leading
    dissipation divergence is balanced by the contact energy released when the
    wetted interval shrinks.
\item[(c)] A third measure--valued regime is conjectured to arise when the
  contact angle diverges even faster in which case the contact line recedes on a
  much smaller time scale, and the mass removed from the support accumulates at
  the boundary.
\end{itemize}
These findings suggest that vanishing-slip regularizations may select
  several distinct limiting evolutions. Even though the large apparent contact
angles in solutions of type (b) and (c) seem incompatible with the lubrication
approximation, a scaling analysis shows that the underlying physical variables
remain consistent with the assumptions for the lubrication approximation.

\bibliographystyle{plain} %
\bibliography{tfe}

\end{document}